\documentclass[12pt,reqno]{amsart}
\usepackage{amssymb}

\usepackage{amsmath,latexsym,amsfonts}
\usepackage{graphicx}
\usepackage{srcltx}
\usepackage{amscd}
\usepackage{hyperref}
\newtheorem{lemma}{Lemma}[section]
\newtheorem{coro}[lemma]{Corollary}
\newtheorem{prop}[lemma]{Proposition}
\newtheorem{thm}[lemma]{Theorem}
\newtheorem{defn}[lemma]{Definition}
\newtheorem{rem}[lemma]{Remark}
\newtheorem{ex}[lemma]{Example}
\makeatletter\@addtoreset{equation}{section}
\renewcommand\theequation{\thesection.\@arabic\c@equation}

\begin{document}

\title[Twisted Basic Cohomology]
{Twisted Basic Dolbeault cohomology on transverse K\"ahler foliations}
%\author[J.~H.~Qian]{Jinhua Qian}
%\address{Department of Mathematics \\
%Northeastern University \\
%Shenyang 110004, China}
%\email[J.~H.~Qian]{qianjinhua@mail.neu.edu.cn}
\author[S.~D.~Jung]{Seoung Dal Jung}
\address{Department of Mathematics\\
Jeju National University \\
Jeju 690-756 \\
Republic of Korea}
\email[S.~D.~Jung]{sdjung@jejunu.ac.kr}


\begin{abstract}
In this paper, we study the twisted basic Dolbeault cohomology  and transverse hard Lefschetz theorem on   a transverse K\"ahler foliation.  And we give some properties for $\Delta_\kappa$-harmonic forms and prove the  Kodaira-Serre type duality and Dolbeault isomorphism for the twisted basic Dolbeault cohomology.
\end{abstract}

\maketitle

 \renewcommand{\thefootnote}{}

\footnote {The author was supported by  the National Research Foundation of Korea (NRF) grant funded by the Korea government (MSIP) (NRF-2018R1A2B2002046).  }

\renewcommand{\thefootnote}{\arabic{footnote}}
\setcounter{footnote}{0}

\section{Introduction}
Let $(M,\mathcal{F})$ be a smooth manifold with a foliation $\mathcal F$. One of smooth invariants of $\mathcal F$ is the basic cohomology. 
The basic forms of $(M,\mathcal F)$ are locally forms on the leaf space; that is, forms $\phi$ satisfying $X\lrcorner\phi = X\lrcorner d\phi =0$ for any vector $X$ tangent to the leaves, where $X\lrcorner$ denotes the interior product with $X$.
Basic forms are preserved by the exterior derivative and are used to define basic de Rham cohomology groups $H_B^* (\mathcal F)$, which is defined by
\begin{align*}
H_B^r(\mathcal F) ={\ker d_B\over {\rm Im}\,d_B},
\end{align*}
where $d_B$ is the restriction of $d$ to the basic forms.
In 
general, the basic de Rham cohomology group does not necessarily satisfy Poincar\'e  duality, in fact, it satisfies the twisted Poincar\'e duality \cite{KT2}: 
\begin{align*}
H_B^r (\mathcal F) \cong H_T^{q-r}(\mathcal F),
\end{align*}
where $q=codim(\mathcal F)$ and $H_T^r(\mathcal F)={\ker d_T\over {\rm Im}\; d_T}$ is the cohomology of $d_T =d_B -\kappa_B\wedge$. Here $\kappa_B$ is the basic part of the mean curvature form $\kappa$ of $\mathcal F$. 

 It is well-known \cite{Al,EKH,KT3,PaRi} that on a compact oriented manifold $M$ with a transversally oriented Riemannian foliation $\mathcal F$, $H_B^r (\mathcal F)\cong {\rm ker}\Delta_B$ is finite dimensional, where $\Delta_B$ is the basic Laplacian.
 Because of this Hodge theorem, researchers have been able to show relationships between curvature bounds and basic cohomology.
 In \cite{Heb},  J. Hebda proved that a lower bound on transversal Ricci curvature for a Riemannian foliation of a compact manifold causes the space of leaves to be compact and the first basic cohomology group to be trivial.
 Another example relating the geometry and the topology came in 1991, when M. Min-Oo et al. \cite{MO} proved that if the transversal  curvature operator of $(M, \mathcal F)$ is positive definite, then the cohomology $H_B^r (\mathcal F)=0 \ (0<r<q)$;  that is, any basic harmonic $r$-form is trivial. 
  There are many other examples of known relationships between transversal curvature and basic cohomology.
 
Recently,  G. Habib and K. Richardson \cite{HR}  introduced the twisted basic de Rham cohomology 
\begin{equation*}
H_\kappa^r(\mathcal F) ={\ker d_\kappa \over {\rm Im}\, d_\kappa}
\end{equation*}
 of the twisted differential operator $d_\kappa = d_B - \frac12 \kappa_B\wedge$ and proved the Poincar\'e duality of the  twisted basic de Rham cohomology $H_\kappa^*(\mathcal F)$   on foliations, that is,
\begin{align*}
H_\kappa^r(\mathcal F) \cong H_\kappa^{q-r}(\mathcal F).
\end{align*}
Moreover,  $H_\kappa^r(\mathcal F)\cong \ker\Delta_\kappa$, where $\Delta_\kappa = d_\kappa\delta_\kappa + \delta_\kappa d_\kappa$ is the twisted basic Laplacian (cf. Section 2.2).  From the Weitzenb\"ock  formula for the twisted basic Laplacian, if the transversal Ricci curvature ${\rm Ric}^Q$ is non-negative and either $\mathcal F$ is nontaut or ${\rm Ric}^Q$ is positive at some point,  then $H_\kappa^1(\mathcal F)=\{0\}$.  Also, the authors \cite{HR} proved that $\mathcal F$ is taut if and only if $H_\kappa^0(\mathcal F)\cong H_\kappa^q(\mathcal F) =\{0\}$ (Tautness theorem).

On a transverse K\"ahler foliation of codimension $q=2n$, the basic Dolbeault cohomology group of $\bar\partial_B$  is defined by
\begin{equation*} 
H^{r,s}_B(\mathcal F)={\ker \bar\partial_B\over {\rm Im}\, \bar\partial_B},
\end{equation*}
where $0\leq r,s \leq n$.
The Kodaira-Serre duality for the basic Dolbeault cohomology  does not necessarily hold unless $\mathcal F$ is taut,
but we exhibit a version of Kodaira-Serre duality \cite{JR1} that
actually does hold in all cases. That is,
\begin{align*}
H_B^{r,s}(\mathcal F)\cong H_T^{n-r,n-s}(\mathcal F),
\end{align*}
where  $H_T^{r,s}(\mathcal F) ={\ker\bar\partial_T\over {\rm Im}\; \bar\partial_T}$  is the Dolbeault cohomology group of $\bar\partial_T$ (cf. Section 3.2). 
Recently, G. Habib and  L. Vezzoni \cite{HV}  studied the basic Dolbeault cohomology  and proved some vanishing theorem for basic Dolbeault cohomology by using the Weitzenb\"ock formula for the twisted basic Laplacian.

In this paper,  we define the twisted basic Dolbeault cohomolgy $H_\kappa^{r,s}(\mathcal F)$ of the twisted  operator $\bar\partial_\kappa$ (cf. Section 3.2) acting on the basic forms of type $(r,s)$ by
\begin{align*}
H_\kappa^{r,s}(\mathcal F) = {\ker \bar\partial_\kappa\over {\rm Im}\, \bar \partial_\kappa}.
\end{align*}
The twisted basic Dolbeault cohomology satisfies  the Kodaira-Serre duality (Theorem 3.11), that is,
\begin{align*}
H_\kappa^{r,s}(\mathcal F)\cong H_\kappa^{n-r,n-s}(\mathcal F).
\end{align*}
Trivially, if $\mathcal F$ is taut, then $H_B^{r,s} (\mathcal F)\cong H_T^{r,s}(\mathcal F)\cong H_\kappa^{r,s}(\mathcal F)$. 
Also, we prove the Hodge decomposition for the twisted basic Dolbeault cohomology (Theorem 3.13) and give
some properties for $\Delta_\kappa$-harmonic forms.  

\begin{rem}
The twisted basic de Rham cohomology is the special cohomology of Lichnerowicz basic cohomology on foliations \cite{Had}.  The Lichnerowicz cohomology is the  cohomology of the complex $(\Omega^*(M),d_\theta)$ of differential forms on a smooth manifold $M$ with the de Rham differential operator $d_\theta=d+\theta\wedge$ deformed by a closed 1-form $\theta$. The Lichnerowicz cohomology is a  proper tool of locally conformal symplectic geometry \cite{HR}.
\end{rem}
\begin{rem} The Lichnerowicz basic cohomology depends only on the basic class of a closed basic 1-form \cite [Proposition 3.0.11]{Had}.   In particular,  the twisted basic de Rham cohomology group depends only on the \'Alvarez class $[\kappa_B]\in H_B^1(\mathcal F)$. That is, for any bundle-like metrics  associated with the same transverse structure, the twisted basic de Rham cohomologies are isomorphic \cite[Theorem 2.11]{HR}. 
 \end{rem}

\section{The basic cohomology}\label{Basic}
\subsection{The basic cohomology}\label{knownfact}
Let $(M,\mathcal F,g_Q)$ be a $(p+q)$-dimensional Riemannian  foliation of codimension $q$ with a holonomy invariant metric $g_Q$ on the normal bundle $Q=TM/T\mathcal F$, meaning that $L_Xg_Q=0$ for all $X\in T\mathcal F$ on a Riemannian manifold $(M,g_M)$  with a bundle-like metric $g_M$ adapted to $g_Q$, where $T\mathcal F$ is the tangent bundle of $\mathcal F$ and $L_X$ denotes the Lie derivative.  
Let $\nabla$ be the transverse Levi-Civita connection on the normal bundle $Q$, which is torsion-free and metric with respect to $g_Q$ \cite{TO1,TO2}. Let $R^Q$ and ${\rm Ric}^Q$ be the transversal curvature tensor and the transversal Ricci operator  of $\mathcal F$ with respect to $\nabla$, respectively.  The mean curvature vector $\tau$ of $\mathcal F$ is given by
\begin{equation*}
\tau=\sum_{i=1}^p \pi(\nabla^M_{f_i}f_i),
\end{equation*}
where $\{f_i\}_{i=1,\cdots,p}$ is a local orthonormal basis of $T\mathcal F$ and $\pi:TM\to Q$ is a natural projection. Then the {\it mean curvature form} $\kappa$
of $\mathcal F$ is given by
\begin{equation*}
\kappa(X)=g_Q(\tau,\pi(X))
\end{equation*}
for any tangent vector $X\in \Gamma TM$.
 Let $\Omega_B^r(\mathcal F)$ be
the space of all {\it basic $r$-forms}, i.e.,  $\phi\in\Omega_B^r(\mathcal F)$ if and only if
$X\lrcorner\phi=0$ and $L_X\phi=0$ for any vector $X\in \Gamma T\mathcal F$.  The foliation $\mathcal F$ is said to be {\it minimal}  if $\kappa=0$. It is well-known \cite{TO1} that on
a compact manifold, $\kappa_B$ is closed, i.e., $d\kappa_B=0$, where $\kappa_B$ is the basic part of $\kappa$.  And the  mean curvature form satisfies the Rummler's formula:
\begin{equation*}\label{Rummlerformula}
d\chi_{\mathcal F} =-\kappa\wedge\chi_{\mathcal F} +\varphi_0,\quad \chi_{\mathcal F}\wedge *\varphi_0=0,
\end{equation*}
where $\chi_{\mathcal F}$ is the characteristic form of $\mathcal F$ and  $*$ is the Hodge star operator associated to $g_M$.

Now we recall the transversal star operator $\bar *:\Omega_B^r(\mathcal F)\to \Omega_B^{q-r}(\mathcal F)$ given by
\begin{align*}
\bar * \phi = (-1)^{p(q-r)}*(\phi\wedge\chi_{\mathcal F}).
\end{align*}
 Trivially, $\bar *^2\phi = (-1)^{r(q-r)}\phi$ for any basic form $\phi\in\Omega_B^r(\mathcal F)$. Let $\nu$ be the transversal volume form, that is, $*\nu=\chi_{\mathcal F}$. Then  the pointwise inner product $\langle\cdot,\cdot\rangle$ on $\Omega_B^*(\mathcal F)$ is defined by $\langle\phi,\psi\rangle\nu = \phi\wedge\bar *\psi$.

 Let $d_B=d|_{\Omega_B^*(\mathcal F)}$ and $d_T=d_B -\kappa_B\wedge$.  
Then the formal adjoint operators $\delta_B$ and $\delta_T$ of  $d_B$ and $d_T$ on basic forms are given by 
\begin{equation}\label{2-1}
\delta_B=(-1)^{q(r+1)+1}\bar *d_T \bar *= \delta_T+\kappa_B^\sharp\lrcorner,\quad
\delta_T=(-1)^{q(r+1)+1}\bar *d_B \bar *, 
\end{equation}
respectively \cite{Al,KT3,TO2}, where  $\kappa_B^\sharp$ be the dual vector of $\kappa_B$.  For a local orthonormal frame $\{E_a\}_{a=1,\cdots,q}$ of the normal bundle $Q$, $\delta_T$ is given \cite{JR2} by 
\begin{equation}\label{2-2}
\delta_T =-\sum_{a=1}^q E_a\lrcorner \nabla_{E_a}.
\end{equation}  
 Now we define two Laplacians $\Delta_B$ and $\Delta_T$ 
acting on $\Omega_B^*(\mathcal F)$ by 
\begin{equation*}
\Delta_B=d_B\delta_B+\delta_B d_B, \quad \Delta_T =d_T\delta_T +\delta_T d_T,
\end{equation*}
respectively. The Laplacian $\Delta_B$ is said to be {\it basic Laplacian}. Then
\begin{equation}\label{2-3}
\Delta_B \bar * = \bar *\Delta_T.
\end{equation} 
Also, we have the generalization of the usual de Rham-Hodge decompositon.
\begin{thm} (\cite{KT3},\cite{TO2})  Let $(M,\mathcal F,g_Q)$ be a transversally oriented Riemannian foliation $\mathcal F$ on  a compact oriented manifold $M$  with a bundle-like metric. Then
\begin{align*}
\Omega_B^r(\mathcal F) &= \mathcal H_B^r(\mathcal F) \oplus {\rm Im}\,d_B \oplus{\rm Im}\,\delta_B\\
&= \mathcal H_T^r(\mathcal F)  \oplus {\rm Im}\,d_T \oplus {\rm Im}\,\delta_T
\end{align*} 
with finite dimensional  $\mathcal H_B^r(\mathcal F)=\ker \Delta_B$ and $\mathcal H_T^r(\mathcal F) =\ker\Delta_T$.
\end{thm}
On a compact manifold,   $d_T^2 =0$ because of  $d\kappa_B =0$. So the basic de Rham cohomology groups $H_B^r(\mathcal F)$ and $H_T^r(\mathcal F)$ are defined  by
\begin{align*}
H_B^r(\mathcal F)={\ker d_B\over {\rm Im}\,d_B},\quad H_T^r(\mathcal F) ={\ker d_T\over {\rm Im}\,d_T},
\end{align*}
respectively.
 Then  it is well-known \cite{TO2}  that 
 \begin{align*}
  H_B^r (\mathcal F)\cong \mathcal H_B^r(\mathcal F),\quad   H_T^r(\mathcal F)\cong \mathcal H_T^r(\mathcal F).
  \end{align*}
  From (\ref{2-3}), we have the twisted duality \cite{KT2}
\begin{align}\label{twistedduality}
 H_B^r (\mathcal F)\cong H_T^{q-r}(\mathcal F).
\end{align}
If the foliation  $\mathcal F$ is  taut (means that there exists a Riemannian metric on $M$ for which all leaves are minimal), then the Poincar\'e duality  holds \cite{KT1}. That is,
\begin{align*}
 H_B^r (\mathcal F)\cong  H_B^{q-r}(\mathcal F). 
 \end{align*}
 Now, we introduce the operator $\nabla_{\rm tr}^*\nabla_{\rm tr}$, which is defined by
 \begin{align*}
 &\nabla_{\rm tr}^*\nabla_{\rm tr}\phi =-\sum_a \nabla_{E_a}\nabla_{E_a}\phi +\nabla_{\kappa_B^\sharp}\phi.
 \end{align*}
 It is well-known (\cite[Proposition 3.1]{JU1}) that the operator $\nabla_{\rm tr}^*\nabla_{\rm tr}$ is non-negative and
formally self-adjoint. Then we have the following  generalized Weitzenb\"ock formula \cite{JU1}:
for any basic form $\phi\in
\Omega_B^r(\mathcal F)$
\begin{equation}\label{2-6}
 \Delta_B \phi = \nabla_{\rm tr}^*\nabla_{\rm tr}\phi +A_{\kappa_B^\sharp}(\phi)+ F(\phi),
\end{equation}
 where  $A_Y(\phi)=L_Y\phi -\nabla_Y\phi$ and
 \begin{align*}
F(\phi)=\sum_{a,b}\theta^a \wedge E_b\lrcorner R^Q(E_b, E_a)\phi.
 \end{align*}
 Here, $\theta^a$ is the dual basic 1-form of $E_a$.
 In particular, if $\phi$ is a basic 1-form, then $F(\phi)^\sharp={\rm Ric}^Q(\phi^\sharp)$.

\begin{rem}
Observe that, from the preceding results, the dimensions of $H_B^r(\mathcal F)$ and $ H_T^r(\mathcal F)$ are smooth invariants of $\mathcal{F}$ and do not depend on the choices of bundle-like metric $g_M$ or on the transversal metric $g_Q$, even though the spaces of harmonic forms do depend on these choices.
\end{rem}
 
\begin{thm} \cite{JR1}
Let $(M,\mathcal F,g_M)$ be a compact Riemannian manifold with a foliation $\mathcal F$ of codimension $q$ and a bundle-like metric $g_M$.
If the endomorphism $F$ is positive-definite, then there are no nonzero basic harmonic forms, that is, $
H_B^r (\mathcal F)=\{0\}.$ In particular, if ${\rm Ric}^Q$ is positive-definite, then $H_B^1(\mathcal F)=\{0\}$.
\end{thm}

\subsection{The twisted basic cohomology}
Now, we recall the twisted differential operators $d_\kappa $ and $\delta_\kappa$  \cite{HR}, which are given by
\begin{align}\label{2-7}
d_\kappa = d_B -\frac12 \kappa_B\wedge, \quad \delta_\kappa =\delta_B -\frac12 \kappa_B^\sharp\lrcorner.
\end{align}
The operator $\delta_\kappa$ is a formal adjoint operator of $d_\kappa$ with respect to the global inner product.
Let  $\Delta_\kappa=d_\kappa \delta_\kappa + \delta_\kappa d_\kappa$ be the {\it twisted basic Laplacian}.  Then we have the following relation: for any basic form $\phi$, 
\begin{align}\label{2-8}
\Delta_\kappa\phi = \Delta_B\phi -\frac12\Big(L_{\kappa_B^\sharp} +( L_{\kappa_B^\sharp})^*\Big)\phi +\frac14|\kappa_B|^2\phi.
\end{align}
Also, we have the following relations.
\begin{lemma} \cite{HR} On $\Omega_B^r(\mathcal F)$, the following equations hold:

\begin{enumerate}
\item \ $ d_\kappa^2 = \delta_\kappa^2 =0,\quad [ d_\kappa,\Delta_\kappa]=[\Delta_\kappa,\delta_\kappa]=[\Delta_\kappa,\bar *]=0$.

\item \  $d_\kappa \bar * = (-1)^{r}\bar * \delta_\kappa,\quad \bar * d_\kappa =(-1)^{r+1}\delta_\kappa \bar *$.
\end{enumerate}

\end{lemma}
Since $d_\kappa^2=0$, we can define the {\it twisted basic de Rham cohomology group} $H_\kappa^r(\mathcal F)$ by
\begin{align*}
H_\kappa^r(\mathcal F) ={\ker d_\kappa\over {\rm Im}\, d_\kappa}.
\end{align*}
Then we have the Hodge decomposition.
\begin{thm} \cite{HR} Let $(M,\mathcal F,g_Q)$ be as in Theorem2.1. Then
\begin{align*}
\Omega_B^r (\mathcal F) = {\mathcal H}_\kappa^r (\mathcal F) \oplus {\rm Im}\, d_\kappa \oplus {\rm Im}\, \delta_\kappa
\end{align*}
with finite dimensional ${\mathcal H}_\kappa^r(\mathcal F)=\ker\Delta_\kappa$. Moreover,  $ {\mathcal H}_\kappa^r (\mathcal F) \cong  H_\kappa^r (\mathcal F)$.
\end{thm}
And  we have the Poincar\'e duality for $d_\kappa$-cohomology. That is,
\begin{thm} \cite{HR} (Poincar\'e duality for $d_\kappa$-cohomology)  On a compact Riemannian manifold with a Riemannian foliation $\mathcal F$, we have
\begin{align*}
H_\kappa^r (\mathcal F)\cong  H_\kappa^{q-r} (\mathcal F).
\end{align*}
\end{thm}
\begin{rem} {\rm Theorem 2.6 resolves the problem of the failure of Poincar\'e duality to hold for standard basic de Rham cohomology $H_B^r(\mathcal F)$ (cf. (\ref{twistedduality}))}.
\end{rem}
\begin{thm} \cite{HR} (Tautness Theorem)  Let $(M,\mathcal F,g_Q)$ be as in Theorem 2.1. Then  $\mathcal F$ is taut if and only if $H_\kappa^0(\mathcal F)\cong H_\kappa^{q}(\mathcal F) \ne\{0\}$.
\end{thm}
 From (\ref{2-6}) and (\ref{2-8}), we have the Weitzenb\"ock formula \cite{HR} for the twisted basic Laplacian $\Delta_\kappa$. Namely,   for any basic form $\phi$,
\begin{align*}
\Delta_\kappa \phi = \nabla_{\rm tr}^*\nabla_{\rm tr}\phi + F(\phi) + \frac14|\kappa_B|^2\phi.
\end{align*}
Then we have the following theorem.
\begin{thm} \cite{HR}  Let $(M,\mathcal F,g_Q)$ be a Riemannian foliation on a  compact, connected manifold $M$ with a bundle-like metric such that the mean curvature form $\kappa$ is basic-harmonic. Then:

(1) if the operator $F + {\frac14}|\kappa|^2$  is strictly positive, then $H_\kappa^r (\mathcal F)=\{0\}$.

(2) if the transversal Ricci curvature ${\rm Ric}^Q$ is non-negative and either $M$ is nontaut or  ${\rm Ric}^Q$ is positive at least one point,  then $H_\kappa^1(\mathcal  F) =\{0\}$.

(3) suppose that the transversal sectional curvatures are nonnegative and positive at least one point. If  $\mathcal F$ is nontaut, then $H_\kappa^r(\mathcal F)=\{0\}$ for $1<r<q$.
\end{thm}

\section{The  twisted basic Dolbealt cohomology }
\subsection{The basic Dolbeault cohomology}
In this section, we generally use the same notation and cite elementary results from \cite[Section 3]{JJ} and \cite{NT}.
Let $(M,\mathcal F,g_Q,J)$ be a  transverse
K\"ahler foliation of codimension $q=2n$ on a Riemannian manifold $M$ with a holonomy invariant transverse Hermitian metric $g_Q$ and an  almost complex structure $J$ 
on $Q$ such that  $\nabla J=0$,
with $\nabla$ being the transversal Levi-Civita connection on $Q$, extended in the usual way to tensors \cite{NT}. 
In some of what follows, we will merely need the fact that the foliation is transverse Hermitian
(all of the above, merely requiring $J$ is integrable and not that $\nabla J=0$), and other times we will 
need the full power of the K\"ahler condition $\nabla J=0$. The  basic K\"ahler form $\omega$ is given by
\begin{equation*}\label{3-1}
\omega(X,Y)=g_Q(\pi(X),J\pi(Y))
\end{equation*}
for any vecgor fields $X, Y\in TM$.
Locally, the basic K\"ahler form $\omega$ may be expressed by
\begin{align*}
\omega = -\frac12\sum_{a=1}^{2n}\theta^a\wedge J\theta^a,
\end{align*}
where $\{\theta^a\}_{a=1,\cdots,2n}$ is a local orthonormal frame of $Q^*$.
Here, we extend $J$ to elements of $Q^*$ by setting $(J\phi)(X)=-\phi(JX)$
for any $X\in Q_x$ and $\phi\in Q_x^*$. 
When it is convenient, we will also refer to the bundle map $J':TM\to TM$ 
defined by $J'(v)=J(\pi(v))$ and 
abuse notation by denoting $J=J'$. Similarly, we sometimes will act on all of 
$T^*M$ using the symbol $J$.
We note that all of the above is true for transverse Hermitian foliations, but the form $\omega$ is not closed
unless the foliation is K\"ahler.

Let $Q^C=Q\otimes\mathbb C$ be the complexified normal bundle and let 
\begin{align*}
Q^{1,0}=\{Z\in Q^C|JZ=iZ\},\quad Q^{0,1}=\{Z\in Q^C | JZ=-iZ\}.
\end{align*}
The element of $Q^{1,0}$ (resp. $Q^{0,1}$) is called a {\it complex normal vector field of type} $(1,0)$ (resp. (0,1)).
Then $Q^C = Q^{1,0} \oplus Q^{0,1}$.
Now, let $Q^*_C$ be the real dual bundle of $Q^C$, defined at each $x\in M$ to be the $\mathbb{C}$-linear maps 
from $Q_x^C$ to $\mathbb{C}$. Then  $Q_C^* =Q_{1,0}\oplus Q_{0,1}$, where
\begin{align*}
Q_{1,0}=\{\theta+i J\theta |\ \theta\in Q^*\}\quad \text{and}\quad 
Q_{0,1}=\{\theta-iJ\theta |\ \theta\in Q^*\}.
\end{align*}
Let $\Lambda^{r,s}_CQ^*$ be the subspace of $\Lambda Q_C^*$ spanned by $\xi\wedge\eta$, where $\xi\in\Lambda^r Q_{1,0}$ and $\eta\in\Lambda^s Q_{0,1}$. The sections of $\Lambda^{r,s}_CQ^*$ are said to be  {\it forms of type $(r,s)$}. Let $\Omega_B^{r,s}(\mathcal F)$ be the set of the basic forms of type $(r,s)$.
Let $\{E_a,JE_a\}_{a=1,\cdots,n}$ be a local orthonormal  frame of $Q$ and $\{\theta^a,J\theta^a\}_{a=1,\cdots,n}$ be their dual basic forms on $Q^*$.  Let $V_a ={1\over\sqrt 2}(E_a -iJE_a)$ and $\omega^a ={1\over\sqrt 2}(\theta^a +iJ\theta^a) $. Then 
\begin{align*}
\omega^a(V_b)=\overline\omega^a(\overline V_b)=\delta_{ab},\ \omega^a(\overline V_b)=\overline\omega^a(V_b)=0.
\end{align*}
A frame fields $\{V_a\}_{a=1,\cdots,n}$ is a local orthonormal  frame of $Q^{1,0}$, which is called a {\it normal frame field of type $(1,0)$}, and $\{\omega^a\}_{a=1,\cdots,n}$ is a dual   frame of $Q_{1,0}$. 

Now, we extend the connection $\nabla$ on $Q$ in the natural way so that $\nabla_XY$ is defined for any $X\in\Gamma (TM\otimes \mathbb{C})$ and any $Y\in \Gamma (Q^C)$.
We further extend it to differential forms, requiring that $\nabla$ is a Hermitian connection, i.e., for any $V\in Q^C$ and any $\phi,\psi\in\Omega_B^{r,s}(\mathcal{F})$,
\begin{align*}
V\langle\phi,\psi\rangle =\langle\nabla_V\phi,\psi\rangle +\langle\phi,\nabla_{\overline V}\psi\rangle.
\end{align*}
It is an easy exercise to show that for any complex vector field $X$, $\nabla_X$ preserves the $(r,s)$ type of the form or vector field.

Now, the transversal star operator $\bar*:\Omega_B^{r,s}(\mathcal F)\to \Omega_B^{n-s,n-r}(\mathcal F)$  on $\Omega_B^{*,*}(\mathcal F)$  is given by
\begin{align*}
\phi\wedge\bar *\bar\psi = \langle\phi,\psi\rangle \nu 
\end{align*}
for any $\phi,\psi \in \Omega_B^{r,s}(\mathcal F)$, where $\nu ={\omega^n\over n!}$ is the transversal volume form. Then for any $\phi\in\Omega_B^{r,s}(\mathcal F)$,
\begin{align*}
\overline{\bar*\phi} =\bar*\bar\phi,\quad \bar*^2 \phi = (-1)^{r+s}\phi.
\end{align*}
From (\ref{2-1}), the adjoint operators $\delta_B $ and $\delta_T$ of $d_B$ and $d_T$  are given  by
\begin{align}\label{deltabtformulas}
\delta_B = - \bar* d_T \bar*, \quad\delta_T = -\bar * d_B\bar *,
\end{align}
respectively.   Note that  $d_B=\partial_B+\bar \partial_B$ and $d_T =\partial_T + \bar\partial_T$, where
 \begin{align}\label{3-1-1}
\partial_T\phi =\partial_B\phi -\kappa_B^{1,0}\wedge\phi,\quad \bar\partial_T\phi =\bar\partial_B\phi -\kappa_B^{0,1}\wedge\phi,
\end{align}
where $\kappa_B^{1,0} = \frac 12 (\kappa_B + i J\kappa_B ) \in\Omega_B^{1,0}(\mathcal{F})$ and $\kappa_B^{0,1} =\overline{\kappa_B^{1,0}}$ \cite{JR1}.  

Let  $\partial_T^*$, $\bar\partial_T^*$, $\partial_B^*$ and $\bar\partial_B^*$ be the formal adjoint operators of $\partial_T,\ \bar\partial_T,\ \partial_B$ and $\bar\partial_B$, respectively, on the space of basic forms. Then 
  \begin{align}\label{3-2}
  \delta_B =\partial_B^* +\bar\partial_B^*,\quad \delta_T =\partial_T^* +\bar\partial_T^*,
  \end{align}
  and from (\ref{deltabtformulas}),
\begin{align}
&\partial_T^*\phi = -\bar *\bar\partial_B\bar *\phi,\quad \bar\partial_T^*\phi =-\bar *\partial_B\bar *\phi, \label{starBarDeltaForm1} \\
&\partial_B^* \phi = -\bar * \bar\partial_T \bar *\phi,\quad \bar\partial_B^* \phi =-\bar *\partial_T\bar*\phi. \label{starBarDeltaForm2} 
\end{align}
Since $\bar * (\kappa_B^{0,1}\wedge) \bar * = H^{1,0}\lrcorner$  \cite{JR1},  
from (\ref{3-1-1}) and (\ref{starBarDeltaForm2}), we have \cite{JR1}
\begin{align}\label{3-5}
\partial^*_B \phi =\partial_T^*\phi +H^{1,0}\lrcorner\,\phi,\quad 
\bar\partial_B^*\phi=\bar\partial_T^*\phi +H^{0,1}\lrcorner\,\phi,
\end{align} 
where $H^{1,0}=\frac 12(\kappa_B^\sharp-iJ\kappa_B^\sharp)$ and $ 
H^{0,1} = \overline{H^{1,0}}$.  Then from (\ref{2-2}) and (\ref{3-2})
\begin{align}\label{3-6}
\partial_T^*\phi = -\sum_{a=1}^n V_a\lrcorner \nabla_{\overline V_a}\phi,\quad \bar\partial_T^*\phi=  -\sum_{a=1}^n \overline V_a\lrcorner \nabla_{V_a}\phi.
\end{align} 
Since $\bar\partial_B^2 =0$,  we can define the {\it basic Dolbeault cohomology group} $H_B^{r,s}(\mathcal F)$ by
\begin{equation*}
H_B^{r,s}(\mathcal F) ={{\ker \bar\partial_B}\over {\rm Im}\:\bar\partial_B}.
\end{equation*}
Now, let  $\square_B =\partial_B\partial_B^* +\partial_B^*\partial_B$  and $\overline\square_B =\bar\partial_B\bar\partial_B^* +\bar\partial_B^*\bar\partial_B$. 
Then we have the basic Dolbeault decomposition.
\begin{thm} \cite{EK,JR1}  Let $(M,\mathcal F,g_Q,J)$ be a transverse K\"ahler foliation on a compact Riemannian manifold $M$ with a bundle-like metric.  Then
\begin{equation*}
\Omega_B^{r,s}(\mathcal F) = \mathcal H_B^{r,s}(\mathcal F)  \otimes {\rm Im}\, \bar \partial_B \otimes {\rm Im} \,\bar\partial_B^*,
\end{equation*}
where $\mathcal H_B^{r,s}(\mathcal F)=\ker\overline\square_B$ is finite dimensional. Moreover, $\mathcal H_B^{r,s}(\mathcal F)\cong H_B^{r,s}(\mathcal F)$.
\end{thm}
 Generally,  the basic Laplacians do not satisfies the properties which hold on a ordinary  K\"ahler  manifold such as $\Delta = 2\square =2\overline\square$. But if $\mathcal F$ is taut, then $\Delta_B = 2\square_B = 2\overline\square_B$ \cite{JR1}. 

\subsection{The twisted basic Dolbeault cohomology}
Let   $\partial_\kappa: \Omega_B^{r,s}(\mathcal F)\to \Omega_B^{r+1,s}(\mathcal F)$ and $\bar{\partial}_\kappa:\Omega_B^{r,s}(\mathcal F)\to \Omega_B^{r,s+1}(\mathcal F)$ be defined  by
\begin{align}\label{3-7}
\partial_\kappa\phi = \partial_B\phi -\frac12  \kappa_B^{1,0}\wedge\phi,\quad
\bar{\partial}_\kappa\phi=\bar\partial_B\phi -\frac12 \kappa_B^{0,1}\wedge\phi,
\end{align}
respectively.  From (\ref{2-7}), it is trivial that  $d_\kappa =\partial_\kappa + \bar{\partial}_\kappa.$ 
Let  $\partial_\kappa^*$ and $\bar{\partial}_\kappa^*$ be  the formal adjoint operators of $\partial_\kappa$ and $\bar\partial_\kappa$, respectively. Then  we have the following.
\begin{prop}  On a transverse K\"ahler foliation, we have
\begin{align*}
   \partial_\kappa^* =\partial_B^* -\frac12 H^{1,0}\lrcorner,\quad
   \bar{\partial}_\kappa^*=\bar\partial_B^* -\frac12 H^{0,1}\lrcorner,\quad \delta_\kappa = \partial_\kappa^* +\bar\partial_\kappa^*.
   \end{align*}
   \end{prop}
   \begin{proof}  From (\ref{2-7}) and (\ref{3-7}),  the proofs  are easy.
   \end{proof}
Let $L:\Omega_B^r(\mathcal F)\to \Omega_B^{r+2}(\mathcal F)$ and $\Lambda:\Omega_B^r(\mathcal F)\to \Omega_B^{r-2}(\mathcal F)$ be given by
\begin{align*}
L(\phi)=\omega\wedge\phi,\quad\Lambda(\phi)=\omega\lrcorner\phi,
\end{align*}
respectively, where $(\xi_1 \wedge\xi_2)\lrcorner=\xi_2^\sharp\lrcorner \xi_1^\sharp\lrcorner$ for any basic 1-forms $\xi_i (i=1,2)$. 
 Trivially, $\langle L\phi,\psi\rangle =\langle\phi,\Lambda\psi\rangle$ and $\Lambda=-\bar*L\bar*$ \cite{CW}.  Also,  it is well-known  \cite{JJ} that 
\begin{align}\label{3-8}
[L,X\lrcorner]=JX^b\wedge,\quad [\Lambda,X^b\wedge]=- JX\lrcorner,\quad
[L,X^b\wedge]=[\Lambda,X\lrcorner]=0
\end{align}
for any vector field $X\in Q$.
From  (\ref{3-8}), we have the following.
\begin{prop} \cite{JJ} On a transverse K\"ahler foliation, we have 
\begin{align*}
& [L,d_B]=[\Lambda,\delta_B]=[L,\partial_B]=[L,\bar\partial_B] =[\Lambda,\partial_B^*]=[\Lambda,\bar\partial_B^*]=0,\\
&[L,\partial_B^*]=-i\bar\partial_T,\ [L,\bar\partial_B^*]=i\partial_T,\ [\Lambda,\partial_B]=-i\bar\partial_T^*,\ [\Lambda,\bar\partial_B]=i\partial_T^*.
\end{align*}
\end{prop}
From  (\ref{3-8}) and Proposition 3.3, we have the following.
\begin{prop} On a  transverse K\"ahler foliation, we have 
\begin{align}
&[L, d_\kappa]=[\Lambda,\delta_\kappa]=[L,\partial_\kappa]=[L,\bar{\partial}_\kappa]=[\Lambda,\partial_\kappa^*]=[\Lambda,\bar{\partial}_\kappa^*]=0,\label{3-9}\\
&[L,\partial_\kappa^*]=-i\bar{\partial}_\kappa,\ [L,\bar{\partial}_\kappa^*]=i \partial_\kappa,\ [\Lambda,\partial_\kappa] =-i\bar{\partial}_\kappa^*,\ [\Lambda,\bar{\partial}_\kappa]=i\partial_\kappa^*.\label{3-10}
\end{align}
\end{prop}
Let $\square_\kappa$ and $\overline\square_\kappa$
be  Laplace operators, which are defined by
\begin{align*}
\square_\kappa =\partial_\kappa \partial_\kappa^* + \partial_\kappa^*\partial_\kappa \quad{\rm and}\quad\overline\square_\kappa =  \bar\partial_\kappa\bar\partial_\kappa^* + \bar\partial_\kappa^*\bar\partial_\kappa,
\end{align*}
respectively. Trivially, $\square_\kappa$ and $\overline\square_\kappa$ preserve the types of the forms.
\begin{thm}  On a transverse K\"ahler foliation, we have
\begin{align*}
\square_\kappa = \overline\square_\kappa,\quad\Delta_\kappa= 2\square_\kappa =2\overline\square_\kappa.
\end{align*}
\end{thm}
\begin{proof}
Since $d_\kappa^2=0$,  it is trivial that $\partial_\kappa^2 =\bar\partial_\kappa^2 =\partial_\kappa\bar\partial_\kappa +\bar\partial_\kappa\partial_\kappa=0$. 
From Proposition 3.4 (\ref{3-10}), we have
\begin{align*}
i(\partial_\kappa\partial_\kappa^* +\partial_\kappa^*\partial_\kappa)&=\partial_\kappa\Lambda\bar\partial_\kappa +\Lambda\bar\partial_\kappa\partial_\kappa -\partial_\kappa\bar\partial_\kappa\Lambda-\bar\partial_\kappa\Lambda\partial_\kappa\\
&=[\partial_\kappa,\Lambda]\bar\partial_\kappa +\bar\partial_\kappa[\partial_\kappa,\Lambda]\\
&=i(\bar\partial_\kappa^*\bar\partial_\kappa +\bar\partial_\kappa\bar\partial_\kappa^*)
\end{align*}
and 
\begin{align*}
i(\bar\partial_\kappa \partial_\kappa^* + \partial_\kappa^* \bar\partial_\kappa) = 
i(\partial_\kappa \bar\partial_\kappa^* + \bar\partial_\kappa^* \partial_\kappa) =0.
\end{align*}
Hence  $\square_\kappa = \overline\square_\kappa$ and  by a direct calculation,
\begin{align*}
\Delta_\kappa =\square_\kappa +\overline\square_\kappa=2\square_\kappa =2\overline\square_\kappa.
\end{align*}
\end{proof}
\begin{rem}  Recall that  if the transverse K\"ahler foliation is minimal, then a basic form of type $(r,0)$ is basic-harmonic if and only if it is basic holomorphic \cite{JR1}. But, if the foliation is not minimal, then the relation does not hold.
\end{rem}

\begin{defn} {\rm On a transverse K\"ahler foliation, a basic form $\phi$ is said to be}  $\bar\partial_\kappa$-holomorphic {\rm if $\bar\partial_\kappa\phi=0$.}
\end{defn} 

\begin{thm} On a  transverse K\"ahler foliation,  a $\bar\partial_\kappa$-holomorphic form of type $(r,0)$ is $\Delta_\kappa$-harmonic. In addition, if $M$ is compact, then the converse holds.
\end{thm}
\begin{proof} Let $\phi$ be a $\bar\partial_\kappa$-holomorphic form of type $(r,0)$.   Since  $\bar\partial^*_\kappa\phi=0$ automatically, $\bar\partial_\kappa\phi$ implies $\Delta_\kappa\phi=2\overline\square_\kappa\phi=0$. Conversely, if $M$ is compact,  then $\Delta_\kappa\phi=0$ implies that $ \int_M |\bar\partial_\kappa\phi|^2=0$, i.e., $\phi$ is $\bar\partial_\kappa$-holomorphic. 
\end{proof}
Now, we consider  $\bar\partial_\kappa$-complex
\begin{align*}
\cdots \overset{\bar\partial_\kappa}\longrightarrow\Omega_B^{r,s-1}(\mathcal F)\overset{\bar\partial_\kappa}\longrightarrow\Omega_B^{r,s}(\mathcal F)\overset{\bar\partial_\kappa}\longrightarrow\Omega_B^{r,s+1}(\mathcal F)\overset{\bar\partial_\kappa}\longrightarrow\cdots.
\end{align*} 
Since $\bar\partial_\kappa^2=0$, the {\it twisted basic Dolbeault cohomology group}  is defined by
\begin{align*}
H_\kappa^{r,s}(\mathcal F)={{\ker\bar\partial_\kappa}\over {\rm Im}\: \bar\partial_\kappa}.
\end{align*}
Then we have the generalization of the Dolbeault decomposition.
\begin{thm} Let $(M,\mathcal F,g_Q,J)$ be a transverse K\"ahler foliation  on a compact manifold $M$ with a bundle-like metric. Then
\begin{align*}
\Omega_B^{r,s}(\mathcal F)= \mathcal H_\kappa^{r,s}(\mathcal F) \oplus {\rm Im}\,\bar\partial_\kappa \oplus {\rm Im}\,\bar\partial_\kappa^*,
\end{align*}
where $\mathcal H_\kappa^{r,s}(\mathcal F)=\ker\overline\square_\kappa$ is finite dimensional. 
\end{thm}
\begin{proof}
The proof is similar to the one in Theorem 2.2. See \cite{KT3} precisely. 
\end{proof}
As a Corollary of Theorem 3.9, we have the Dolbeault isomorphism.
\begin{coro} (Dolbeault isomorphism)  Let $(M,\mathcal F,g_Q,J)$ be as in Theorem 3.9. Then
\begin{equation*}
\mathcal H_\kappa^{r,s} (\mathcal F)\cong H_\kappa^{r,s}(\mathcal F).
\end{equation*}
\end{coro}
\begin{proof} The proof is similar to the proof of the Hodge isomorphism.
\end{proof}
 Then we have the Kodaira-Serre duality.
\begin{thm} (Kodaira-Serre duality)  Let $(M,\mathcal F,g_Q,J)$ be as in Theorem 3.9. Then
\begin{align*}
H_\kappa^{r,s}(\mathcal F)\cong H_\kappa^{n-r,n-s}(\mathcal F).
\end{align*}
\end{thm}
\begin{proof}
We define the operator $\sharp:\Omega_B^{r,s}(\mathcal F)\to \Omega_B^{n-r,n-s}(\mathcal F)$ by
\begin{align*}
\sharp\phi :=\bar *\bar\phi,
\end{align*}
which is an isomorphism.  
 Since  $\bar * (\kappa_B^{1,0}\wedge) \bar * = H^{0,1}\lrcorner$ \cite{JR1}, we have that  for $\phi\in\Omega_B^{r,s}(\mathcal F)$, 
\begin{equation}\label{3-15}
\bar *(\kappa_B^{0,1}\wedge \phi) = (-1)^{r+s} H^{1,0}\lrcorner \bar *\phi
\end{equation} 
 and from (\ref{starBarDeltaForm1}), we have
 \begin{equation}\label{3-16}
 \bar *\bar\partial_B\phi = (-1)^{r+s+1} \partial_T^* \bar *\phi.
 \end{equation}
 From (\ref{3-15}) and (\ref{3-16}), we get  that on $\Omega_B^{r,s}(\mathcal F)$,
 \begin{equation}\label{3-17}
 \bar *\bar\partial_\kappa = (-1)^{r+s+1}\partial_\kappa^* \bar *.
 \end{equation}
 Also, from (\ref{starBarDeltaForm2}),  we get that on $\Omega_B^{r,s}(\mathcal F)$, 
\begin{equation}\label{3-18}
 \bar *\bar\partial_\kappa^* = (-1)^{r+s}\partial_\kappa\bar *.
\end{equation}
From (\ref{3-17}) and (\ref{3-18}),  we have that on $\Omega_B^{r,s}(\mathcal F)$,
\begin{align*}
\bar *\bar\partial_\kappa \bar\partial_\kappa^*= -\partial_\kappa^*\partial_\kappa\bar *,\quad
\bar * \bar\partial_\kappa^* \bar\partial_\kappa= -\partial_\kappa \partial_\kappa^* \bar *,
\end{align*}
which implies
\begin{align*}
\bar *\overline\square_\kappa = \square_\kappa \bar *.
\end{align*}
Hence  for any basic form $\phi\in \Omega_B^{r,s}(\mathcal F)$,  we get
\begin{align*}
\sharp \overline\square_\kappa\phi = \bar * \square_\kappa \bar \phi= \overline\square_\kappa \bar *\bar\phi=\overline\square_\kappa \sharp\phi.
\end{align*}
That is, $\sharp$ preserves  $\mathcal H_\kappa^{r,s}(\mathcal F)=\ker\overline\square_\kappa$.  From the Dolbeault isomorphism (Corollary 3.10), the proof follows.
\end{proof}

\begin{rem} {\rm 
In general, the Kodaira-Serre duality  does not hold for $\bar\partial_B$-cohomology. In fact, 
\begin{align*}
 H_B^{r,s}(\mathcal F)\cong  H_T^{n-r,n-s}(\mathcal F),
\end{align*}
where  $H_T^{r,s}(\mathcal F)= { {\ker\bar\partial_T}\over {{\rm Im}\, \bar\partial_T}}$ is the $\bar\partial_T$-cohomology.  
$H_T^{r,s}(\mathcal F)$  is a type of
 Lichnerowicz basic cohomology. The interested reader may consult \cite[Section 3]{Ba} and \cite[Section 3, called ``adapted cohomology'' here]{Va} for information about ordinary Lichnerowicz cohomology and \cite{Had} for the basic case. 
Theorem 3.11 resolves the problem of the failure of Kodaira-Serre duality to hold for $\bar\partial_B$-cohomology.}
\end{rem}

From Theorem 3.5, we have the following Hodge decomposition for the twisted basic Dolbeault cohomology.
\begin{prop}  Let $(M,\mathcal F,g_Q,J)$ be as in Theorem 3.9. Then
\begin{equation}\label{3-19}
H_\kappa^l(\mathcal F) = \oplus_{r+s=l}H_\kappa^{r,s}(\mathcal F)
\end{equation}
for $0\leq l\leq 2n$  and
\begin{equation}\label{3-20}
{\rm dim}_{\mathbb C} H_\kappa^{r,s}(\mathcal F)  = {\rm dim}_{\mathbb C} H_\kappa^{s,r}(\mathcal F).
\end{equation}
\end{prop}
\begin{proof}  
Since $\Delta_\kappa = 2\overline\square_\kappa$ and $\overline\square_\kappa$ preserves the space  $\Omega_B^{r,s}(\mathcal F)$,   the proof of (\ref{3-19}) follows by Hodge isomorphism (Theorem 3.11).  The proof of  (\ref{3-20}) follows from that  the map $\mathcal H_\kappa^{r,s}(\mathcal F)\to  \mathcal H_\kappa^{s,r}(\mathcal F)$ is a conjugate linear isomorphism.  
\end{proof}

\begin{rem}   The Hodge decomposition  for $\bar\partial_B$-cohomology does not hold unless the mean curvature of $\mathcal F$ is automorphic \cite[Corollary 6.7]{JR2}.
\end{rem} 

Now let $\Omega_{B,P}^r(\mathcal F)$ be the set of all {\it primitive} basic $r$-forms $\phi$, that is, $\Lambda\phi=0$.  Then by $\frak{sl}_2(\mathbb C)$ representation theory, we have the following proposition.
\begin{prop}  \cite{JR2} Let $(M,\mathcal F,g_Q,J)$ be as in Theorem 3.9.  Then we have the following.

(1)  $\Omega_{B,P}^r(\mathcal F) = 0$ if $r>n$.

(2) If $\phi\in\Omega_{B,P}^r(\mathcal F)$, then $L^s\phi \ne 0$ for $0\leq s\leq n-r$ and $L^s\phi=0$ for $s>n-r$.

(3) The map $L^s :\Omega_B^r(\mathcal F)\to \Omega_B^{r+2s}(\mathcal F)$ is injective for $0\leq s\leq n-r$.

(4) The map $L^s : \Omega_B^r (\mathcal F)\to \Omega_B^{r+2s}(\mathcal F)$ is surjective for $s\geq n-r$.

(5) $\Omega_B^r (\mathcal F) = \oplus_{s\geq 0} L^s \Omega_{B,P}^{r-2s}(\mathcal F)$.

\end{prop}

\begin{thm} (Hard Lefschetz theorem)   Let $(M,\mathcal F,g_Q,J)$ be a transverse K\"ahler foliation on a compact manifold $M$ with a bundle-like metric. Then the Hard Lefschetz theorem holds for twisted basic  cohomology. That is, the map
\begin{equation}\label{3-21}
L^s:H_\kappa^{r}(\mathcal F)\to H_\kappa^{r+2s}(\mathcal F)
\end{equation}
is injective for $0\leq s \leq n-r$ and surjective for $s\geq n-r$, $s\geq 0$. Moreover,
\begin{align}
H_\kappa^r(\mathcal F) &= \oplus_{s\geq 0} L^s H_{\kappa,P}^{r-2s}(\mathcal F),\label{3-22}\\
H_\kappa^{r,s}(\mathcal F)&= \oplus_{t\geq 0} L^t H_{\kappa,P}^{r-t,s-t}(\mathcal F),\label{3-23}
\end{align}
where $H_{\kappa,P}^r(\mathcal F) \cong \Omega_{B,P}^r(\mathcal F)\cap \ker\Delta_\kappa$ and $H_{\kappa,P}^{r,s}(\mathcal F) \cong \Omega_{B,P}^{r,s}(\mathcal F)\cap \ker\Delta_\kappa$.
\end{thm}
\begin{proof}  Since $\bar\partial_\kappa \partial_\kappa + \partial_\kappa\bar\partial_\kappa=0$,  $[L,\overline\square_\kappa]=[L,\square_\kappa]=0$  by Proposition 3.4, and so $[L,\Delta_\kappa]=0$.  Hence  by Proposition 3.15 and Hodge isomorphism (Theorem 2.5), the proofs of (\ref{3-21})  and (\ref{3-22}) follow.  The proof of (\ref{3-23}) follows from the Dolbeault isomorphism (Corollary 3.10).
\end{proof}
\begin{rem} Generally, Hard Lefschetz theorem for basic  cohomology does not hold unless $[\partial_B\kappa_B^{0,1}]$ is trivial. (cf. \cite[Theorem 5.11]{JR2}). 
\end{rem}
\begin{ex}  We consider the Carri\`{e}re example from 
\cite{Ca}. Also, see \cite[Section 7.1]{HR} and \cite[Example 9.1]{JR2}.  Let $A$ be a matrix in $\mathrm{SL}_{2}(\mathbb{Z})$ of trace
strictly greater than $2$. We denote respectively by $w_{1}$ and $w_{2}$
unit eigenvectors associated with the eigenvalues $\lambda $ and $\frac{1}{%
\lambda }$ of $A$ with $\lambda >1$ irrational. Let the hyperbolic torus $%
\mathbb{T}_{A}^{3}$ be the quotient of $\mathbb{T}^{2}\times \mathbb{R}$ by
the equivalence relation which identifies $(m,t)$ to $(A(m),t+1)$, where $\mathbb T^2 =\mathbb R^2/\mathbb Z^2$ is the torus. The flow
generated by the vector field $W_{2}$ is a Riemannian foliation with
bundle-like metric (letting $\left( x,s,t\right) $ denote local coordinates
in the $w_{2}$ direction, $w_{1}$ direction, and $\mathbb{R}$ direction,
respectively) 
\begin{equation*}
g=\lambda ^{-2t}dx^{2}+\lambda ^{2t}ds^{2}+dt^{2}.
\end{equation*}%
Since $A$ preserves the integral lattice $\mathbb Z^2$, it induces a diffeomorphism $A_0$ of the torus $\mathbb T^2$. So the flow generated by $W_2$ is invariant under the diffeomorphism $A_0$ of $\mathbb T^2$. 
Note that the mean curvature of the flow is $\kappa =\kappa _{B}=\log \left(
\lambda \right) dt$, since $\chi _{\mathcal{F}}=\lambda ^{-t}dx$ is the
characteristic form and $d\chi _{\mathcal{F}}=-\log \left( \lambda \right)
\lambda ^{-t}dt\wedge dx=-\kappa \wedge \chi _{\mathcal{F}}$. It is well known that  all twisted de Rham cohomology groups vanish, that is, $H_\kappa^r(\mathcal F)=\{0\}$ for all $r=0,1,2$ in \cite{HR}.  

Now we will show that all twisted basic Dolbeault cohomology grous satisfy Kodaira-Serre duality. First, we note that  an
orthonormal frame field for this manifold is $\{X=\lambda ^{t}\partial
_{x},S=\lambda ^{-t}\partial _{s},T=\partial _{t}\}$ corresponding to the
orthonormal coframe $\{X^{\ast }=\chi _{\mathcal{F}}=\lambda ^{-t}dx,S^{\ast
}=\lambda ^{t}ds,T^{\ast }=dt\}$. Then, letting $J$ be defined by $%
J(S)=T,J(T)=-S$, the Nijenhuis tensor 
\begin{equation*}
N_{J}(S,T)=[S,T]+J\left( [JS,T]+[S,JT]\right) -[JS,JT]
\end{equation*}%
clearly vanishes, so that $J$ is integrable. 
The corresponding transverse K\"{a}hler form is seen to be $\omega =T^{\ast
}\wedge S^{\ast }=\lambda ^{t}dt\wedge ds=d(\frac{1}{\log \lambda }S^{\ast
}) $, an exact form in basic cohomology. From the above, 
\begin{equation*}
\kappa _{B}=-i\left( \log \lambda \right) Z^{\ast }+i\left( \log \lambda
\right) \bar{Z}^{\ast },
\end{equation*}%
where $Z^{\ast }=\frac{1}{2}(S^{\ast }+iT^{\ast}) \in \Omega_B^{1,0}(\mathcal F)$. Then
\begin{eqnarray*}
\kappa _{B}^{1,0} &=&-i\log \left( \lambda \right) Z^{\ast }=-\frac{i}{2}%
\left( \log \lambda \right) \left( \lambda ^{t}ds+idt\right) \\
\bar{\partial}_{B}\kappa _{B}^{1,0}& =&d\kappa _{B}^{1,0}=
\left( \log \lambda \right) ^{2}\bar{Z}^{\ast }\wedge Z^{\ast }\\
\bar\partial_\kappa\kappa_B^{1,0} &=&\bar\partial_B\kappa_B^{1,0}-\frac12\kappa_B^{0,1}\wedge\kappa_B^{1,0}=\frac98\bar\partial_B \kappa_B^{1,0}.
\end{eqnarray*}%
It is impossible to change the metric so that $\bar\partial_B\kappa_B^{1,0}=0$. Hence $\mathcal F$ is nontaut \cite[Corollary 5.12]{JR2}.

The basic Dolbeault cohomolog groups  are given by $H_{B}^{0,0}=\mathbb R,\  H_{B}^{1,0}=\{0\}, \  H_{B}^{0,1}=\mathbb R$ and $H_{B}^{1,1}=\{0\}$.
Then observe that the ordinary basic cohomology Betti numbers for this
foliation are $h_{B}^{0}=h_{B}^{1}=1$, $h_{B}^{2}=0$, we see that the basic
Dolbeault Betti numbers satisfy 
\begin{equation*}
h_{B}^{0,0}=h_{B}^{0,1}=1,\quad h_{B}^{1,0}=h_{B}^{1,1}=0.
\end{equation*}%
So even though it is true that 
\begin{equation*}
h_{B}^{j}=\sum_{r+s=j}h_{B}^{r,s},
\end{equation*}%
and the foliation is transversely K\"{a}hler, we also have (with $n=1$) 
\begin{equation*}
h_{B}^{r,s}\neq h_{B}^{s,r},\quad h_{B}^{r,s}\neq h_{B}^{n-r,n-s}.
\end{equation*}%
Thus, for
a nontaut, transverse K\"{a}hler foliation, it is not necessarily true that
the odd basic Betti numbers are even, and the basic Dolbeault numbers do not
have the same kinds of symmetries as Dolbeault cohomology on K\"{a}hler
manifolds.

Now we compute the twisted basic Dolbeault cohomology groups $H_\kappa^{\ast ,\ast }\left( \mathcal{F}\right) $.  Let $f\in H_\kappa^{0,0}(\mathcal F)$,  that is, $f$ is a periodic function alone $t$ and   $\bar\partial_\kappa f =0$. Equivalently,
\begin{align}\label{3-34}
\bar\partial_B f = \frac12 f\kappa_B^{0,1}.
\end{align}
On the other hand,  since $\bar\partial_B f \in\Omega_B^{0,1}(\mathcal F)$,  by a direct calculation,  $\bar\partial_B f = i f'(t) \overline Z ^*$. Hence   from (\ref{3-34}),  $f'(t)\overline Z^*  =\frac12 f (\log \lambda) \overline Z^*$. That is, $f' =\frac12 (\log \lambda)f$. Then $f = c\lambda^{-t}$ for some constant. Since $f$ is periodic, $f(t)=0$.  Hence 
\begin{align*}
H_\kappa^{0,0}(\mathcal F) =\{0\}.
\end{align*}
Let $\varphi\in H_\kappa^{1,0}$. That is, $\varphi = f(t) Z^*\in\Omega_B^{1,0}(\mathcal F)$  and $\bar\partial_\kappa\varphi=0$ for a periodic function $f$. Hence
\begin{align}\label{3-35}
\bar\partial_B\varphi =\frac12\kappa_B^{0,1}\wedge\varphi.
\end{align}
By a direct calculation,  $\bar\partial_B\varphi = {i\over 2}(\log\lambda)f \overline Z^*\wedge Z^*$.  So from (\ref{3-35})
\begin{align*}
f' =-\frac32 (\log \lambda)f
\end{align*}
 and so $f(t) = c\lambda^{-{3\over2 }t}$ for some $c\in\mathbb R$. Since $f$ is periodic, it is zero. Thus, $\varphi =0$. That is,
 \begin{align*}
 H_\kappa^{1,0}(\mathcal F) =\{0\}.
 \end{align*} 
By complex conjugatation,  we get
\begin{align*}
H_\kappa^{0,1}(\mathcal F)=\{0\}.
\end{align*}
Let $\varphi\in H_\kappa^{1,1}(\mathcal F)$. Then $\varphi\in\Omega_B^{1,1}(\mathcal F)$ is of the form $\varphi = f(t) Z^*\wedge\overline Z^*$, where $f$ is a periodic function. Trivially, $\bar\partial_\kappa\varphi =0$. Now, let $\bar\partial_\kappa^*\varphi=0$. Since
\begin{align*}
\delta_\kappa\varphi &= \delta_T\varphi +\frac12 \kappa^\sharp\lrcorner\varphi\\
&= -{i\over 2}(f'-\frac34(\log\lambda)f)(Z^*+\overline Z^*),
\end{align*}
we have $\bar\partial_\kappa^*\varphi = -{i\over 2}(f'-\frac34(\log\lambda)f)Z^*$. Thus  the solution of $f'-\frac34(\log\lambda)f =0$ reduced to zero for periodic functions $f$. Hence
\begin{align*}
H_\kappa^{1,1}(\mathcal F) =\{0\}.
\end{align*}
This shows that the Kodaira-Serre duality and Hodge isomorphism are satisfied for the twisted basic Dolbeault cohomology.
\end{ex}

\subsection{The $d_\kappa d_\kappa^c$ Lemma}
Let $(M,\mathcal F,g_Q,J)$ be a transverse K\"ahler foliation  of codimension $2n$ on a compact Riemannian manifold $M$ 
with bundle-like metric.  First, we recall the operator $C:\Omega_B^*(\mathcal F)\to \Omega_B^*(\mathcal F)$ defined by \cite{JR2}
\begin{equation*}
C =\sum_{0\leq r,s\leq n} (\sqrt{-1})^{r-s} P_{r,s},
\end{equation*}
where $P_{r,s} :\Omega_B^*(\mathcal F)\to \Omega_B^{r,s}(\mathcal F)$ is the projection. Then
\begin{equation*}
C^* = C^{-1} = \sum_{0\leq r,s\leq n} (\sqrt{-1})^{s-r}P_{r,s}.
\end{equation*}
Now, we define $d_\kappa^c : \Omega_B^r(\mathcal F)\to \Omega_B^{r+1}(\mathcal F)$ by
\begin{equation*}
d_\kappa^c = C^* d_\kappa C = C^{-1}d_\kappa C.
\end{equation*}
Then $d_\kappa^c =\sqrt{-1} (\bar\partial_k -\partial_k)$ and 
\begin{equation*}
d_\kappa d_\kappa^c|_{\Omega_B^*(\mathcal F)} = -d_\kappa^c d_\kappa|_{\Omega_B^*(\mathcal F)}.
\end{equation*}
Let   $\delta_\kappa^c$  be  the adjoint operator of $d_\kappa^c$, which is given by
\begin{equation*}
\delta_\kappa^c = C^*\delta_\kappa C = C^{-1}\delta_\kappa C.
\end{equation*}
Let  $\Delta_\kappa^c = d_\kappa^c \delta_\kappa^c + \delta_\kappa^c d_\kappa^c$.
Since $\Delta_\kappa$ preserves the type of differential form, we have
\begin{equation*}
\Delta_\kappa^c = C^{-1}\Delta_\kappa C =\Delta_\kappa.
\end{equation*}
\begin{lemma}  ($d_\kappa d_\kappa^c$ Lemma)  Let $(M,\mathcal F,g_Q,J)$ be a transverse K\"ahler foliation on a compact manifold $M$ with a bundle-like metric.  Then on $\Omega_B^*(\mathcal F)$,
\begin{equation*}
\ker d_\kappa \cap {\rm Im}\; d_\kappa^c = {\rm Im}\; d_\kappa d_\kappa^c.
\end{equation*}
\end{lemma}
\begin{proof}  Let $\alpha\in \ker d_\kappa \cap {\rm Im}\; d_\kappa^c\cap \Omega_B^r(\mathcal F)$.  That is, for some basic $(r-1)-$form $\beta$,  $\alpha = d_\kappa^c\beta$ and $\beta=\gamma + d_\kappa \gamma_1 + \delta_\kappa \gamma_2$ with $\gamma \in \mathcal H_\kappa^{r-1}(\mathcal F)$ by the Hodge decomposition (Theorem 2.5).  Since $M$ is compact, by Theorem 3.5, $\Delta_\kappa \gamma =0$ implies $\square_\kappa\gamma =\overline\square_\kappa\gamma=0$. Therefore, $\partial_\kappa\gamma =\bar\partial_\kappa\gamma=0$ and so $d_\kappa^c\gamma=0$.  Hence
\begin{align}
\alpha &=d_\kappa^c\beta \notag\\&=d_\kappa^c \gamma + d_\kappa^c d_\kappa\gamma_1 + d_\kappa^c\delta_\kappa\gamma_2\notag\\
&= d_\kappa^c d_\kappa\gamma_1 + d_\kappa^c\delta_\kappa\gamma_2\notag\\
&=d_\kappa d_\kappa^c (-\gamma_1) + d_\kappa^c\delta_\kappa\gamma_2.\label{3-24}
\end{align}
Moreover, since $d_\kappa\alpha=0$, by the equation above, 
\begin{align*}
0=d_\kappa d_\kappa^c\beta = d_\kappa d_\kappa^c\delta_\kappa\gamma_2 = -d_\kappa\delta_\kappa d_\kappa^c\gamma_2.
\end{align*}
The last equality follows from  $d_\kappa^c\delta_\kappa +\delta_\kappa d_\kappa^c =0$ (Theorem 3.5).  By integrating,
\begin{align*}
0=\int_M \langle d_\kappa\delta_\kappa d_\kappa^c\gamma_2,d_\kappa^c\gamma_2\rangle = \int_M \Vert \delta_\kappa d_\kappa^c\gamma_2\Vert^2 =\int_M \Vert d_\kappa^c \delta_\kappa\gamma_2\Vert^2.
\end{align*}
That is, $d_\kappa^c\delta_\kappa\gamma_2=0$.  So from (\ref{3-24}), 
\begin{align*}
\alpha = d_\kappa d_\kappa^c(-\gamma_1),
\end{align*}
which implies that $\alpha \in {\rm Im}\; d_\kappa d_\kappa^c$.   
\end{proof}
\begin{rem}  If $\mathcal F$ is taut, then the  $d_\kappa d_\kappa^c$ Lemma implies that $d d_c$ Lemma \cite[Lemma 7.3]{JR}. 
\end{rem}
\subsection{$\Delta_\kappa$-harmonic forms}
Let $(M,\mathcal F,g_Q,J)$  be a transverse K\"ahler foliation  on a compact Riemannian manifold $M$
with a bundle-like metric.  We define two operators
\begin{align*}
\nabla_T^*\nabla_T\phi&=-\sum_a \nabla_{V_a}\nabla_{\bar V_a}\phi + \nabla_{H^{0,1}}\phi,\\
\bar\nabla_T^*\bar\nabla_T\phi&=-\sum_a \nabla_{\bar V_a}\nabla_{V_a}\phi + \nabla_{H^{1,0}}\phi.
\end{align*}
Then by a direct calculation, we have
\begin{align}
\nabla_T^*\nabla_T\phi=\bar\nabla_T^*\bar\nabla_T\phi +\nabla_{H^{0,1}-H^{1,0}}\phi -\sum_a R^Q( V_a,\bar V_a)\phi
\label{DeltaTSquareForm}
\end{align}
for any basic form $\phi$.   Then the operators $ \nabla_T^*\nabla_T$ and $\bar\nabla_T^*\bar\nabla_T$ are formally self-adjoint  and positive-definite \cite{JR1}.

\begin{prop} \cite{JR1}
Let $(M,\mathcal F,g_Q,J)$  be a transverse K\"ahler foliation
on a  Riemannian manifold $M$
with a bundle-like metric.   Then  for all $\phi\in\Omega_B^{r,s}(\mathcal{F})$,
\begin{align}
\overline\square_B\phi&= \nabla_T^*\nabla_T\phi + \sum_{a,b}\bar\omega^a\wedge \bar V_b\lrcorner  R^Q(V_b,\bar V_a)\phi+\sum_a \bar\omega^a\wedge (\nabla_{\bar V_a}H^{0,1})\lrcorner\,\phi,  \label{boxBar}
\end{align} 
\end{prop}
From (\ref{DeltaTSquareForm}), we have the following.
\begin{prop}  \cite{JR1} Let  $(M,\mathcal F,g_Q,J)$  be as in Proposition 3.21. 

$(1)$ If $\phi$ is a basic form of type $(r,0)$, then
\begin{align}
\overline\square_B\phi &=\nabla_T^*\nabla_T\phi.\label{boxBarT}\\
&=\bar\nabla_T^*\bar\nabla_T\phi +\nabla_{H^{0,1}-H^{1,0}}\phi -\sum_a R^Q( V_a,\bar V_a)\phi.
\end{align}

$(2)$ If $\phi$ is a basic form of type $(r,n)$, then
\begin{align}
\overline\square_B\phi&=\nabla_T^*\nabla_T\phi +\sum_a R^Q(V_a,\bar V_a)\phi
 +\operatorname{div}_\nabla (H^{0,1}) \phi \label{boxBarrn1}   \\
 &= \bar\nabla_T^*\bar\nabla_T \phi + \nabla_{H^{0,1}-H^{1,0}}\phi +\operatorname{div}_\nabla (H^{0,1}) \phi.
  \label{boxBarrn2}  
\end{align}

\end{prop}
On the other hand, by a direct calculation, we have the following.

\begin{prop}  On a transverse  K\"ahler foliation, we have
\begin{align*}
\overline\square_\kappa &=\overline\square_B -\frac12\Big( \epsilon(\kappa_B^{0,1})\bar\partial_B^* + \bar\partial_B^* \epsilon(\kappa_B^{0,1})\Big)
-\frac12\Big(\bar\partial_B H^{0,1}\lrcorner + H^{0,1}\lrcorner\bar\partial_B\Big) +\frac12|\kappa_B^{0,1}|^2.
\end{align*}

\end{prop}
From Proposition 3.21 and Proposition 3.23, we have the following.
\begin{prop} On a transverse K\"ahler foliation, we have
\begin{align*}
\overline\square_\kappa\phi&= \nabla_T^*\nabla_T\phi+\sum_{a,b}\overline\omega^a\wedge\overline V_b\lrcorner R^Q(V_b,\overline V_a)\phi + \sum_a \overline\omega^a\wedge (\nabla_{\overline V_a} H^{0,1})\lrcorner\phi \\
&-\frac12\Big( \epsilon(\kappa_B^{0,1})\bar\partial_B^* + \bar\partial_B^* \epsilon(\kappa_B^{0,1})\Big)\phi
-\frac12\Big(\bar\partial_B H^{0,1}\lrcorner + H^{0,1}\lrcorner\bar\partial_B\Big)\phi +\frac12|\kappa_B^{0,1}|^2\phi.
\end{align*}
\end{prop}
\begin{rem} Proposition 3.24  was also shown in \cite[Theorem 3.1]{HV}, but the expression is  little bit different. The authors in \cite{HV} proved the vanishing theorem of transversally holomorphic basic $(r,0)$-form  by using the Weitzenb\"ock formula for the twisted basic Laplacian $\overline\square_\kappa$ \cite[Theorem 3.4]{HV}.  In this research, we  deal with the vanishing theorem of $\bar\partial_\kappa$-holomorphic $(r,0)$-form (Corollary 3.31 below).
\end{rem}

 \begin{prop} \cite{JR2}  Let $(M,\mathcal F,g_Q,J)$ be a transverse K\"ahler foliation on a closed Riemannian manifold $M$. Then there exists a bundle-like metric compatible with the    K\"ahler structure such that
$\partial_B^*\kappa_B^{1,0} =0$, or $\bar\partial_B^*\kappa_B^{0,1}=0$.
\end{prop}
\begin{lemma} For any $\phi\in\Omega_B^{r,0}(\mathcal F)$, we get
\begin{align*}
\bar\partial_B^*\epsilon(\kappa_B^{0,1})\phi = -\nabla_{H^{1,0}}\phi.
\end{align*}
\end{lemma}
\begin{proof}  From (\ref{3-15}) and (\ref{3-16}),  for any $\phi\in\Omega_B^{r,0}(\mathcal F)$, since $ \overline   V_a \lrcorner \phi =0$, we have
\begin{align*}
 \bar\partial_B^*\epsilon(\kappa_B^{0,1})\phi&=-\sum_{a} \overline V_a\lrcorner \nabla_{V_a}\kappa_B^{0,1}\wedge\phi-\sum_a \kappa_B^{0,1}(\overline V_a)\nabla_{V_a}\phi  +  H^{0,1}\lrcorner \epsilon(\kappa_B^{0,1})\phi\\
 &=(\bar\partial_T^*\kappa_B^{0,1})\wedge\phi -\nabla_{H^{1,0}}\phi + |\kappa_B^{0,1}|^2\phi.
 \end{align*}
 By Proposition 3.26,  if we choose the bundle-like metric such that $\bar\partial_B^*\kappa_B^{0,1}=0$, the the proof follows.  
  \end{proof}
 From  Lemma 3.26 and  Proposition 3.27, we get
\begin{thm}  On  a transverse K\"ahler foliation, the following hold:

(1)  If  $\phi\in\Omega_B^{r,0}(\mathcal F)$, then
\begin{align}
\overline\square_\kappa\phi&= \nabla_T^*\nabla_T\phi-\frac12\bar\partial_B^* \epsilon(\kappa_B^{0,1})\phi -\frac12H^{0,1}\lrcorner\bar\partial_B\phi+ \frac12 |\kappa_B^{0,1}|^2\phi \label{3-30}\\
&=\overline\nabla_T^*\overline\nabla_T\phi+\sum_a R^Q(\overline V_a,V_a)\phi +\nabla_{H^{0,1}}\phi
+\frac12\bar\partial_B^* \epsilon(\kappa_B^{0,1})\phi - \frac12 H^{0,1}\lrcorner\bar\partial_B\phi + \frac12 |\kappa_B^{0,1}|^2\phi. \label{3-31}
\end{align}

(2) If $\phi\in\Omega_B^{r,n}(\mathcal F)$, then
\begin{align}
\overline\square_\kappa\phi&= \overline\nabla_T^*\overline\nabla_T\phi-\nabla_{H^{1,0}}\phi +\frac12\nabla_{H^{0,1}}\phi + \frac12 {\rm div}_\nabla(H^{0,1})\phi-\frac12\kappa_B^{0,1}\wedge\bar\partial_B^*\phi + \frac12 |\kappa_B^{0,1}|^2\phi \label{3-30-1}\\
&=\nabla_T^*\nabla_T\phi+\sum_a R^Q(V_a,\overline V_a)\phi -\frac12\nabla_{H^{0,1}}\phi-\frac12\kappa_B^{0,1}\wedge\bar\partial_B^*\phi +\frac12 {\rm div}_\nabla(H^{0,1})\phi + \frac12 |\kappa_B^{0,1}|^2\phi. \label{3-31-1}
\end{align}
\end{thm}
\begin{proof}  Let $\phi$ be a basic form of type $(r,0)$.  Since $Z\lrcorner\phi=0$ for any $Z\in Q^{0,1}$, from Proposition 3.24,  (\ref{3-30}) is proved.   
%\begin{align}   \label{3-32}
%\overline\square_\kappa\phi&= \nabla_T^*\nabla_T\phi-\frac12\bar\partial_B^*\epsilon(\kappa_B^{0,1})\phi -\frac12 H^{0,1}\lrcorner\bar\partial_B\phi + \frac12 |\kappa_B^{0,1}|^2\phi
%\end{align}
Now, we choose the bundle-like metric such that the mean curvature form is basic harmonic, that is, 
$\bar\partial_B^*\kappa_B^{0,1}=0$.  From Lemma 3.26 and (\ref{DeltaTSquareForm}), the proof of (\ref{3-31}) follows.  From Proposition 3.21(2) and Proposition 3.22,  the proofs of (\ref{3-30-1}) and (\ref{3-31-1})  follow.
\end{proof}

\begin{prop}  Let $(M,\mathcal F,g_Q,J)$ be a transverse K\"ahler foliation on a compact manifold $M$ with a bundle-like metric.  If $\phi\in\Omega_B^{r,0}(\mathcal F)$ is a  $\Delta_\kappa$-harmonic form, then
\begin{align*}
\nabla_{V}\phi=0
\end{align*}
for any $V\in\Gamma Q^{0,1}$. 
%Moreover, if $\mathcal F$ is nontaut, then there are no a non-zero $\Delta_\kappa$-harmonic forms of type $(r,0)$, that is,  $\mathcal H_\kappa^{r,0}(\mathcal F)=\{0\}$.
\end{prop}
\begin{proof}  Since $\Delta_\kappa\phi=0$, by Theorem 3.5, $\overline\square_\kappa\phi=0$ and so $\bar\partial_\kappa\phi=0$.  That is, $\bar\partial_B\phi =\frac12\kappa_B^{0,1}\wedge\phi$.  
 Since $\phi$ is type of $(r,0)$,   we have
\begin{align*}
H^{0,1}\lrcorner \bar\partial_B\phi = \frac12 H^{0,1}\lrcorner (\kappa_B^{0,1}\wedge\phi) =\frac12 |\kappa_B^{0,1}|^2\phi.
\end{align*}
From (\ref{3-30}), we get
\begin{align}\label{3-34}
\nabla_T^*\nabla_T\phi -\frac12 \bar\partial_B^*\epsilon(\kappa_B^{0,1})\phi + \frac14 |\kappa_B^{0,1}|^2\phi =0.
\end{align}
By integrating $(\ref{3-34})$, we get
\begin{align}\label{3-35}
0&=\int_M \Vert \nabla_T\phi\Vert^2 -\frac12\int_M \langle\kappa_B^{0,1}\wedge\phi,\bar\partial_B\phi\rangle +\frac14\int_M |\kappa_B^{0,1}|^2\Vert\phi\Vert^2\notag\\
&=\int_M \Vert \nabla_T\phi\Vert^2 -\frac14\int_M \Vert\kappa_B^{0,1}\wedge\phi\Vert^2 +\frac14\int_M |\kappa_B^{0,1}|^2\Vert\phi\Vert^2.
\end{align}
Since $H^{0,1}\lrcorner\phi=0$ for any  $\phi\in\Omega_B^{r,0}(\mathcal F)$, we get 
\begin{align*}
\langle\kappa_B^{0,1}\wedge\phi,\kappa_B^{0,1}\wedge\phi\rangle = \langle\phi, H^{0,1}\lrcorner (\kappa_B^{0,1}\wedge\phi)\rangle
= |\kappa_B^{0,1}|^2 \Vert\phi\Vert^2.
\end{align*}
From (\ref{3-35}), we get 
\begin{align*}
\int_M \Vert \nabla_T\phi\Vert^2 =0,
\end{align*}
which implies  $\nabla_T\phi=0$, that is, for any $V\in\Gamma Q^{0,1}$, $\nabla_V\phi=0$.  
\end{proof}

\begin{thm} \label{VanishingThm}
Let $(M,\mathcal F,g_Q,J)$ be a transverse K\"ahler foliation on a compact manifold $M$ with a bundle-like metric.  

(1) If  the transverse Ricci curvature is nonnegative and positive at some point, then  $\mathcal H_\kappa^{r,0}(\mathcal F)=\{0\}$ for all $r>0$.  

(2) If $\mathcal F$ is transversely Ricci-flat and non-taut, then  $\mathcal H_\kappa^{r,0}(\mathcal F)=\{0\}$ for all $r>0$.  

\end{thm}
\begin{proof}
Let $\phi$ be a $\Delta_\kappa$-harmonic form of type $(r,0)$. If we choose the bundle-like metric such that $\bar\partial_B^*\kappa_B^{0,1}=0$, then  from (\ref{DeltaTSquareForm}), Lemma 3.27 and Proposition 3.29, 
\begin{align}\label{3-36}
\bar\nabla_T^*\bar\nabla_T\phi +\bar\partial_B^*\epsilon(\kappa_B^{0,1})\phi -\sum_a R^Q(V_a,\bar V_a)\phi=0.
\end{align}
Note that $(\sum_a R^Q(\bar V_a,V_a))\omega^b =Ric^Q(E_b,E_b)\omega^b$ \cite[Remark 4.7]{JR1}.
By integrating  (\ref{3-36}), we get
\begin{align*}
\int_M \Vert \bar\nabla_T\phi\Vert^2 + \int_M |\kappa_B^{0,1}|^2\Vert\phi\Vert^2 +\sum_{i=1}^r\int_M R^Q(E_{a_i},E_{a_i})\Vert\phi\Vert^2 =0.
\end{align*}
If $Ric^Q$ is nonnegative and positive at some point, then $\phi=0$. So the proof of (1) is proved.  If  $Ric^Q =0$, then  
\begin{align*}
|\kappa_B^{0,1}|\Vert\phi\Vert=0.
\end{align*}
So if $\mathcal F$ is nontaut, then $\phi=0$, that is,  the proof of (2) is finished.  
\end{proof}

\begin{coro} 
Let $(M,\mathcal F,g_Q,J)$ be as in Theorem~\ref{VanishingThm} with a transversely Ricci-flat foliation. If $\mathcal F$ is nontaut,  then 
\begin{align*}
H^{r,0}_\kappa(\mathcal F)\cong H_\kappa^{0,r}(\mathcal F) \cong H_\kappa^{n,s}(\mathcal F) \cong H_\kappa^{s,n}(\mathcal F)=\{0\}
\end{align*}
for $ r,s>0$.
\end{coro}
\begin{proof}  The proofs  follow from complex conjugation,  Kodaira-Serra duality and Dolbeault isomorphism.
\end{proof}

The vanishing theorems for  the basic harmonic space were proved in  \cite{HV} and \cite{JR1}, respectively.

\end{document}